\documentclass[12pt]{amsart}
\usepackage{amssymb,amsmath}
\def\CC {{\mathbb C}}
\def\L {{\mathbb L }}
\def\H {{\mathcal H}}

\def\V {{\mathcal V}}
\def\A {{\mathbb A}}

\def\R {\mathbb{R}}

\def\D {{\mathfrak D}}

\def\Re {\mathfrak{Re\,}}
\def\Im {\mathfrak{Im\,}}
\def\eps{\varepsilon}
\def\e{{\rm e}}
\def\d{{\rm d}}
\def\ddt{\frac{\d}{\d t}}
\def\i{{\rm i}}
\def \l {\langle}
\def \r {\rangle}
\def \and {{\qquad\text{and}\qquad}}

\newtheorem{proposition}{Proposition}[section]
\newtheorem{theorem}[proposition]{Theorem}
\newtheorem*{thm}{Theorem}

\newtheorem{lemma}[proposition]{Lemma}

\theoremstyle{definition}

\newtheorem*{remark}{Remark}
\numberwithin{equation}{section}
\def \au {\rm}
\def \ti {\it}
\def \jou {\rm}
\def \bk {\it}
\def \no#1#2#3 {{\bf #1} (#3), #2.}
\def \eds#1#2#3 {#1, #2, #3.}

\title[Abstract systems of Timoshenko type]
{Stability analysis of abstract\\ systems of Timoshenko type}

\author[V. Danese, F. Dell'Oro and V. Pata]
{Valeria Danese, Filippo Dell'Oro and Vittorino Pata}

\address{Politecnico di Milano - Dipartimento di Matematica ``F.\ Brioschi''
\newline\indent
Via Bonardi 9, 20133 Milano, Italy}
\email{valeria.danese@polimi.it {\rm (V. Danese)}}
\email{vittorino.pata@polimi.it {\rm (V. Pata)}}

\address{Institute of Mathematics of the Academy of Sciences of the Czech Republic
\newline\indent
\v{Z}itn\'a 25, 115 67 Praha 1, Czech Republic}
\email{delloro@math.cas.cz {\rm (F. Dell'Oro)}}


\begin{document}

\begin{abstract}
We consider an abstract system of Timoshenko type
$$
\begin{cases}
\rho_1{{\ddot \varphi}} + a A^{\frac12}(A^{\frac12}\varphi + \psi) =0\\
\rho_2{{\ddot \psi}} + b A \psi + a (A^{\frac12}\varphi + \psi)  -  \delta A^\gamma {\theta} = 0\\
\rho_3{{\dot \theta}} + c A\theta + \delta A^\gamma {{\dot \psi}} =0
\end{cases}
$$
where the operator $A$ is strictly positive selfadjoint.
For any fixed $\gamma\in\R$, the stability properties of the related solution semigroup $S(t)$
are discussed. In particular, a general technique is introduced in order to prove
the lack of exponential decay
of $S(t)$ when the spectrum of the leading operator $A$ is not made by eigenvalues only.
\end{abstract}

\maketitle

\section{Introduction}

\noindent
Let $(H,\l \cdot,\cdot \r, \|\cdot\| ) $ be an infinite-dimensional separable real Hilbert space,
and let
$$A:\D(A)\subset H\to H$$
be a strictly
positive (real) selfadjoint unbounded linear operator with domain $\D(A)$,
i.e.\ $A=A^*>0$,
where the dense embedding $\D(A)\subset H$ need not be compact.
For $t>0$, we consider the abstract evolutionary system
\begin{equation}
\begin{cases}
\label{SYS}
\rho_1{{\ddot \varphi}} + a A^{\frac12}(A^{\frac12}\varphi + \psi) =0,\\
\rho_2{{\ddot \psi}} + b A \psi + a (A^{\frac12}\varphi + \psi)  -  \delta A^\gamma {\theta} = 0, \\
\rho_3{{\dot \theta}} + c A\theta + \delta A^\gamma {{\dot \psi}} =0,
\end{cases}
\end{equation}
in the unknowns
$\varphi=\varphi(t)$, $\psi=\psi(t)$ and $\theta=\theta(t)$,
where the {\it dot} stands for derivative with respect to the time variable $t$.
Here, the coupling exponent $\gamma$ is a real number,
whereas $\rho_1,\rho_2,\rho_3,$ as well as $a,b,c$ and the coupling parameter $\delta$ are
strictly positive fixed constants.

\begin{remark}
For the particular choice $H=L^2(0,\ell)$ and
$$A=-\partial_{xx}\qquad\text{with domain}\qquad
\D(A)=H^2(0,\ell)\cap H^1_0(0,\ell),
$$
system~\eqref{SYS} can be interpreted as a nonlocal version
of a thermoelastic beam model of Timoshenko type \cite{TIM},
subject to the Dirichlet boundary conditions
$$\varphi(0,t)=\varphi(\ell,t)=\psi(0,t)=\psi(\ell,t)=\theta(0,t)=\theta(\ell,t)=0,
$$
where
$\varphi,\psi,\theta: \,(x,t)\in [0,\ell]\times [0,\infty) \mapsto \R$
represent the transverse displacement of a beam with reference configuration $[0,\ell]$,
the rotation angle of a filament and the relative temperature, respectively.
\end{remark}

For every fixed $\gamma\in\R$, system~\eqref{SYS} is shown to generate a
contraction semigroup $S(t)=\e^{t L}$
acting on the natural weak energy space $\H$.
The focus of this work is a detailed analysis of the stability properties of $S(t)$,
which turn out to depend heavily on the particular choice of $\gamma$.
To this end, a crucial object is the so called stability number, firstly introduced in~\cite{MRR,SOU},
$$
\chi = \frac{a}{\rho_1}-\frac{b}{\rho_2},
$$
defined as the difference between the propagation speeds of the first two
hyperbolic equations.

\smallskip
The main result of the paper can be stated as follows:

\begin{thm}
The semigroup $S(t)$ is exponentially stable if and only if
$$\chi=0 \and \gamma = \frac12.$$
\end{thm}

Possibly, the most interesting feature of the theorem
is that we are not assuming the compactness of the embedding $\D(A)\subset H$.
This translates into the fact that the spectrum of $A$ can be a complicated object,
not simply made of an increasing sequence $\alpha_n\to\infty$ of eigenvalues.
Indeed, the usual semigroup techniques employed
to prove the lack of exponential decay of linear semigroups rely
in a crucial way on the existence of such a sequence $\alpha_n$. Here, we establish a
general method for the exponential stability analysis when the spectrum of $A$
consists of approximate eigenvalues (which is always the case for selfadjoint operators).
This method allows to revisit stability results
for a number of equations or systems, already known when
the leading operator has compact inverse.
For instance, a complete answer can be given
on the uniform decay of the wave equation coupled with the classical heat equation (see e.g.\ \cite{ABB,LZ})
$$
\begin{cases}
\ddot u+ Au - A^\gamma \theta =0,\\
\dot \theta + A \theta + A^\gamma \dot u =0,
\end{cases}
$$
depending on the parameter $\gamma\in\R$.
In this case, exponential stability occurs if and only if $\gamma\in[\frac12,1]$.

\smallskip
Aside from exponential (or uniform) decay, we are also interested in weaker notions of stability
(see \S3 for the definitions). Indeed, we will prove the following theorem.

\begin{thm}
For a general strictly
positive selfadjoint operator $A$,
\begin{itemize}
\item[{\rm (i)}] $S(t)$ is semiuniformly stable for $\gamma\in[\tfrac12,1]$;
\smallskip
\item[{\rm (ii)}] $S(t)$ is not semiuniformly stable when $\gamma>1$.
\end{itemize}
\end{thm}

Conclusion (ii) above follows from the fact that the infinitesimal generator $L$
of $S(t)$ turns out to be invertible
if and only if $\gamma\leq 1$.

In the particular case when $\D(A)\Subset H$ (i.e.\ the embedding $\D(A)\subset H$ is compact),
we are able to complete the picture. Namely, we have

\begin{thm}
If in addition $\D(A)\Subset H$, then
\begin{itemize}
\item[{\rm (i)}] $S(t)$ is semiuniformly stable if and only if $\gamma\leq 1$;
\smallskip
\item[{\rm (ii)}] $S(t)$ is stable for every $\gamma\in\R$.
\end{itemize}
\end{thm}

It is worth noting that, in this situation,
the domain of $L$ is compactly embedded into the phase space $\H$
if and only if $\gamma<1$. Thus, $\gamma=1$
can be regarded as a sort of critical exponent for the problem.

\smallskip
In order to explain the difficulties encountered in the analysis, we begin to observe that
the sole dissipation present in the system is the ``thermal" one, provided by the third (heat) equation,
while the first two (wave) equations alone are conservative.
Accordingly, the stabilization mechanism is based on the transfer of thermal dissipation into
mechanical dissipation, and this happens through the coupling.
Roughly speaking, the coupling should be sufficiently strong in order for the third equation
to transfer enough dissipation, but not too strong due to the first two equations.
Indeed, wave equations with very strong damping are less likely to stabilize (see e.g.\ \cite{CT1,FAT}).
In addition, for systems of Timoshenko type, the stability number $\chi$ comes into play:
when the waves exhibit different speeds, the dissipation transfer from the variable $\psi$ to $\varphi$
loses effectiveness, whereas $\chi=0$ produces a sort of resonance,
as already observed in~\cite{MRR,SOU}.

\subsection*{Plan of the paper}
In the next \S\ref{Sdue} we introduce the functional setting of the problem.
In \S\ref{Stre} we give a general presentation of the decay properties of bounded linear semigroups.
In~\S\ref{Squattro} we rewrite system~\eqref{SYS} as an ODE
by introducing the linear operator $L$, which is proved
to be the infinitesimal
generator of a contraction semigroup in the subsequent \S\ref{Scinque}.
The remaining \S\ref{Ssei},\ref{Ssette},\ref{Sotto},\ref{Snove}
are devoted to the statements and the proofs of the
decay results. In particular, the exponential stability
and the lack of exponential stability of $S(t)$ are discussed in \S\ref{Sotto} and
\S\ref{Snove}, respectively, whereas \S\ref{Ssei} and \S\ref{Ssette} are concerned with stability
and semiuniform stability.

\section{Functional Setting}
\label{Sdue}

\noindent
We consider the nested family of Hilbert spaces
$$
H^r=\D(A^{\frac{r}2}),\quad r\in\R,
$$
with inner products and norms given by
$$
\l u,v\r_r=\l A^{\frac{r}2}u,A^{\frac{r}2}v\r
\and
\|u\|_r=\|A^{\frac{r}2}u\|.
$$
The index $r$ will be always omitted whenever zero.
For $r>0$, it is understood that $H^{-r}$ denotes the completion of the domain, so that
$H^{-r}$ is the dual space of $H^r$.
The symbol $\l \cdot,\cdot \r$ will also be used to denote the duality pairing between
$H^{-r}$ and $H^{r}$.

\begin{remark}
If $u\in H$ and $r>0$, we can still write $A^r u$ to mean the element of the dual space  $H^{-2r}$ acting as
$$\l A^r u,v\r=\l u,A^r v\r, \quad\forall v\in H^{2r}.$$
\end{remark}

Along the paper, we will also consider the complexification of $H$
(and, more generally, the one of $H^r$). This is
the complex Hilbert space
$$H_\CC=H\oplus \i H=\{u+\i v:\, u,v\in H\},$$
endowed with the inner product
$$\l u+\i v, u'+\i v'\r
=\l u,u'\r+\l v,v'\r+\i \l v, u'\r-\i\l u,v'\r.
$$
In a similar manner, we define the
complexification $\A$ of
$A$ to be the linear
operator on $H_\CC$ with domain
$$\D(\A)=\{u+\i v:\,
u,v\in\D(A)\}$$
acting as
$$\A (u+\i v)=Au+\i Av.$$
Since $A$ is strictly positive selfadjoint, so is $\A$, and the two spectra
$\sigma(A)$ and $\sigma(\A)$ coincide. In particular,
$\sigma(\A)\subset\R$, and the resolvent sets satisfy the equality $\rho(A)=\rho(\A)\cap\R$. Besides,
$$\alpha_0=\min\{\alpha:\, \alpha\in\sigma(\A)\}>0$$
and (as $\A$ is unbounded)
$$\sup\{\alpha:\, \alpha\in\sigma(\A)\}=\infty.$$
Being $\A$ selfadjoint, it is well known that $\sigma(\A)$
coincides with the approximate point spectrum $\sigma_{\rm ap}(\A)$, made by approximate eigenvalues
(see e.g.\ \cite{KD}); namely,
$\alpha$ belongs to the spectrum of $\A$ if and only if there is a sequence of unit vectors $w_n\in H_\CC$ such that
$$\lim_{n\to\infty}\|(\alpha-\A)w_n\|= 0.$$
If $\D(A)\Subset H$, which is the same as saying that $A^{-1}$ and
$\A^{-1}$ are compact operators, then $\sigma(\A)$ reduces to
the point spectrum $\sigma_{\rm p}(\A)$, consisting of an increasing sequence $\alpha_n\to\infty$
of eigenvalues of $\A$.

Denoting by $E_\A$ the spectral measure of $\A$ (see e.g.\ \cite{RUDIN}),
for every complex measurable
function $f$ on $\sigma(\A)$ one can define the linear operator
$$
f(\A)=\int_{\sigma(\A)}f(t)\,\d E_\A(t)
$$
with dense domain
$$
\D(f(\A))=\Big\{z\in H_\CC \,:\, \int_{\sigma(\A)}|f(t)|^2\,\d \mu^\A_{z}(t)<\infty\Big\}.
$$
Here, $\mu^\A_z$ is the finite measure on $\CC$ supported on $\sigma(\A)$ given by
$$\mu^\A_z(\Sigma)=\|E_\A(\Sigma)z\|^2,
$$
for every Borel set $\Sigma\subset\CC$.
Recall that $f(\A)$ is selfadjoint if and only if $f$ is real valued.
Furthermore,
$$\|f(\A)z\|^2=\int_{\sigma(\A)}|f(t)|^2\,\d \mu^\A_{z}(t),\quad\forall z\in\D(f(\A)).$$
In particular, for every $r>0$ we deduce the Poincar\'e type inequality
\begin{equation}
\label{POIN}
\|z\|\leq {\alpha_0^{-\frac{r}2}}\|z\|_r,\quad \forall z\in H^r_\CC.
\end{equation}

\section{Decay Types of Bounded Linear Semigroups}
\label{Stre}

\noindent
In this section, we dwell on the possible decay types of
a strongly continuous semigroup of linear operators
$$S(t)=\e^{tL}:\H\to\H$$
acting on a real Hilbert space $\H$,
with infinitesimal generator
$L:\D(L)\subset\H\to\H$. We further assume that $S(t)$ is bounded, i.e.\
$$\sup_{t\geq 0}\|S(t)\|_{{\mathfrak L}(\H)}<\infty,$$
where ${\mathfrak L}(\H)$ denotes the Banach space of bounded linear operators
on $\H$.

\begin{remark}
The complexification $\L$ of the linear operator $L$ is the infinitesimal generator
of the bounded semigroup
$$S_\CC(t)(z+\i w)=S(t)z+\i S(t)w$$
on the complex Hilbert space $\H_\CC=\H\oplus\i\H$, which satisfies
the equality
$$
\|S(t)\|_{{\mathfrak L}(\H)}=\|S_\CC(t)\|_{{\mathfrak L}(\H_\CC)},\quad\forall t\geq 0.
$$
\end{remark}

\subsection{Stability}
The semigroup $S(t)$ is said to be {\it stable} if
$$\lim_{t \to \infty} \| S(t) z\|_\H = 0, \quad \forall z \in \H.$$

\begin{remark}
As far as stability is concerned, there is no need to require the
boundedness of $S(t)$ in the hypotheses. Indeed,
as a consequence of the Uniform Boundedness Principle,
a stable semigroup is automatically bounded.
\end{remark}

For a contraction semigroup,
the following stability criterion from~\cite{ChePat} can be useful.

\begin{theorem}
\label{THM-STABLE}
Let $S(t)$ be a contraction semigroup
(i.e.\ $\|S(t)\|_{{\mathfrak L}(\H)}\leq 1$ for all $t$),
and let $\V \subset \H$ be a Hilbert space with continuous and dense embedding
(but not necessarily compact).
Suppose that, for every fixed $z\in \V$,
\begin{enumerate}
\item[{\rm (i)}] The set $\,\bigcup_{t\geq t_z} S(t)z$ is
relatively compact in $\H$ and bounded in $\V$, for some $t_z \geq 0$;
\vskip1mm
\item[{\rm (ii)}] $\| S(t)z \|_\H = \| z \|_\H$ for all $t > 0$ implies that $z=0$.
\end{enumerate}
Then $S(t)$ is stable.
\end{theorem}

\subsection{Semiuniform stability}
The semigroup $S(t)$ is said to be {\it semiuniformly stable} if there exists a nonnegative function
$h(t)$ vanishing at infinity such that
$$
\|S(t)z\|_\H \leq h(t)\|Lz\|_\H, \quad \forall z \in \D(L).
$$

\begin{remark}
Semiuniform stability is a stronger notion than stability.
Indeed, it ensures the convergence $S(t)z\to 0$ for all $z\in\D(L)$,
and since $S(t)$ is bounded, this immediately yields the convergence
$S(t)z\to 0$ for all $z\in\H$.
\end{remark}

In light of the works of C.J.K.\ Batty and coauthors~\cite{Batty,BattyBis,Batty1},
semiuniform stability can be given equivalent formulations,
as shown in the next theorem, which also provides an effective criterion.

\begin{theorem}
\label{THM-SUSTABLE}
The following are equivalent:
\begin{enumerate}
\item[{\rm (i)}] $S(t)$ is semiuniformly stable.
\smallskip
\item[{\rm (ii)}] The imaginary axis $\i\R$ belongs to the resolvent set $\rho(\L)$.
\smallskip
\item[{\rm (iii)}] $\L$ is invertible\footnote{Since $\L$ is a closed operator,
being the infinitesimal generator of a (complex) strongly continuous semigroup,
by the Closed Graph Theorem we learn that $\L^{-1}\in{\mathfrak L(\H_\CC})$.} and
$\,\lim_{t\to \infty} \|S_\CC(t){\L}^{-1}\|_{{\mathfrak L}(\H_\CC)} = 0$.
\smallskip
\item[{\rm (iv)}] $\lim_{t \to \infty}{\| S_\CC(t)(1-\L)^{-1}\|}_{{\mathfrak L}(\H_\CC)} = 0$.
\end{enumerate}
\end{theorem}

\begin{proof}
The equivalence (ii)$\Leftrightarrow$(iv) is proved in \cite{Batty1},
while (ii)$\Leftrightarrow$(iii) can be found in~\cite{Batty,BattyBis}.
It is also apparent that (iii)$\Rightarrow$(i). We conclude the proof by showing the implication
(i)$\Rightarrow$(iv).
To this end, we first observe that (i) gives at once
$$\|S_\CC(t)z\|_{\H_\CC} \leq h(t)\|\L z\|_{\H_\CC}, \quad \forall z \in \D(\L).$$
Since $S_\CC(t)$ is a bounded semigroup, $1\in\rho(\L)$.
Accordingly,
$$\|S_\CC(t)(1-\L)^{-1}z \|_{\H_\CC} \leq h(t)\|\L(1-\L)^{-1}z\|_{\H_\CC}, \quad \forall z \in \H_\CC.$$
On the other hand,
$$\L(1-\L)^{-1}\in {\mathfrak L}(\H_\CC),$$
and so there exists $M\geq 0$ such that
$$\|S_\CC(t)(1-\L)^{-1}z \|_{\H_\CC} \leq h(t) M \|z\|_{\H_\CC}, \quad \forall z \in \H_\CC,$$
yielding the desired limit (iv).
\end{proof}

\subsection{Exponential or uniform stability}
The semigroup $S(t)$ is said to be {\it exponentially stable} (or {\it uniformly stable})
if there exist $K\geq 1$ and $\kappa>0$ such that
$$
\|S(t)\|_{{\mathfrak L}(\H)} \leq K\e^{-\kappa t}.
$$
Since the infinitesimal generator of an exponentially
stable semigroup is always invertible,
it is apparent that exponential stability implies semiuniform stability:
just take
$$h(t)=K\|L^{-1}\|_{{\mathfrak L}(\H)}\e^{-\kappa t}.$$

\begin{remark}
It is well know that exponential stability
occurs if and only if
$$\lim_{t \to \infty} \| S(t)\|_{{\mathfrak L}(\H)} = 0,$$
which is true if and only if
$$\| S(t_\star)\|_{{\mathfrak L}(\H)}<1,
$$
for some $t_\star>0$ (see e.g.\ \cite{PAZ}). Actually, by the Uniform Boundedness Principle,
exponential stability can be inferred whenever there exists
a nonnegative function
$k(t)$ vanishing at infinity such that
$$
\|S(t)z\|_\H \leq C_z k(t)\|z\|_\H, \quad \forall z \in \H,
$$
where the positive constant $C_z$ depends on $z$. In other words, lack of exponential stability prevents
the existence of a uniform decay pattern of the trajectories.
\end{remark}

As shown by J.\ Pr\"uss \cite{PRU}, the exponential stability of a semigroup,
no matter if bounded or not, is equivalent to the condition
$$\i\R\subset\rho(\L)
\and
\sup_{\lambda\in\R}\|(\i\lambda - \L)^{-1}\|_{{\mathfrak L}(\H_\CC)} < \infty.
$$
When (as in the present case) $S(t)$ is bounded, the result can be given
a more convenient formulation~\cite{GNP}.

\begin{theorem}
\label{THM-EXPSTABLE}
The semigroup $S(t)$ is exponentially stable if and only if
there exists $\eps>0$ such that
$$
\inf_{\lambda\in \R}\|\i\lambda z -\L z\|_{\H_\CC}\geq \eps\|z\|_{\H_\CC},\quad \forall z \in \D(\L).
$$
\end{theorem}

\section{The Linear Operator $L$}
\label{Squattro}

\noindent
We define the phase space of our problem to be
$$\H = H^1 \times H \times H^1 \times H \times H$$
endowed with the (equivalent) Hilbert norm
$$\| (\varphi, {\tilde \varphi}, \psi, \tilde \psi, \theta) \|_{\H}^2 =
a\| A^{\frac12}\varphi + \psi\|^2 +
b\| \psi\|^2_1 + \rho_1\| \tilde \varphi\|^2 +
\rho_2\| \tilde \psi\|^2 + \rho_3\| \theta\|^2.$$
Then, introducing the evolution
$$Z(t) = (\varphi(t), {\tilde \varphi}(t), \psi(t), {\tilde \psi}(t), \theta(t)),$$
we rewrite system~\eqref{SYS} as the ODE in $\H$
$$\ddt Z(t) = LZ(t),$$
where the linear operator $L$ is given by
$$L
\begin{pmatrix}
\varphi\\
\tilde \varphi\\
\psi\\
\tilde \psi\\
\theta
\end{pmatrix} =
\begin{pmatrix}
{\tilde \varphi}\\
-\frac{a}{\rho_1}A^{\frac12}(A^{\frac12}\varphi + \psi) \\
\tilde \psi\\
\frac{1}{\rho_2} A(\delta A^{\gamma-1} {\theta} - b \psi ) -\frac{a}{\rho_2} (A^{\frac12}\varphi + \psi) \\
-\frac{1}{\rho_3}A(c\theta + \delta A^{\gamma-1} {\tilde \psi})
\end{pmatrix},
$$
with domain
$$\D(L) =
\begin{Bmatrix}
 z \in \H \left |
\begin{matrix}
{\tilde \varphi} \in H^1 \\
\varphi \in H^2 \\
{\tilde \psi} \in H^1 \\
\delta A^{\gamma-1} {\theta} - b \psi \in H^2\\
c\theta + \delta A^{\gamma-1} {\tilde \psi} \in H^2
\end{matrix}
\right.
\end{Bmatrix}.
$$

\smallskip
By the very definition of $\D(L)$, some additional
regularity on the components $\tilde\psi$ and $\theta$ is obtained.\footnote{Actually, further regularization
occurs for the remaining variables as well.}

\begin{proposition}
\label{PropXT-reg}
Let $z = (\varphi, \tilde\varphi, \psi, \tilde\psi, \theta) \in \D(L)$. Then
\begin{equation}
\label{tilpsireg}
\tilde \psi \in
\begin{cases}
H^{4\gamma -3}  &\text{{\rm if }}\, 1 < \gamma \le \frac{3}{2},\\
H^{2\gamma} &\text{{\rm if }}\, \gamma > \frac{3}{2},
\end{cases}
\end{equation}
and
\begin{equation}
\label{thetareg}
\theta \in
\begin{cases}
H^2  &\text{{\rm if }}\, \gamma \leq \frac{1}{2},\\
H^{3-2\gamma} &\text{{\rm if }}\, \frac{1}{2} < \gamma \leq 1,\\
H^{2\gamma -1} &\text{{\rm if }}\, \gamma>1.
\end{cases}
\end{equation}
In particular, $\theta \in H^1$ for every $\gamma \in \R$.
\end{proposition}

\begin{proof}
If $\gamma > 1$, since $\psi \in H^1$, we infer from the condition
$$\delta A^{\gamma-1} {\theta} -b \psi \in H^2$$
that $\theta \in H^{2\gamma -1}$. This observation, together with
$$c\theta + \delta A^{\gamma-1} {\tilde \psi} \in H^2,$$
imply~\eqref{tilpsireg}.
Assume next $\gamma \leq 1$. In this case, using once more the latter relation,
$$\tilde \psi \in H^1\quad\Rightarrow\quad A^{\gamma -1} \tilde \psi \in H^{3- 2\gamma}
\quad\Rightarrow\quad \theta \in H^{\min\{2,3- 2\gamma\}},$$
which is exactly~\eqref{thetareg}.
\end{proof}

Further properties of $L$ are established here below.

\begin{proposition}
\label{PropDissipative}
The operator $L$ is dissipative for every $\gamma \in \R$.
\end{proposition}

\begin{proof}
This amounts to show that
\begin{equation*}
{\l L z, z\r}_{\H} \leq 0, \quad \forall z \in \D(L).
\end{equation*}
Indeed, given $z = (\varphi, \tilde\varphi, \psi, \tilde\psi, \theta) \in \D(L)$,
direct computations yield
\begin{equation}
\label{Real-DISSIPATIVO}
\l L z, z\r_{\H} = - c\| \theta\|^2_1,
\end{equation}
where cancelations are allowed due to the regularity of the domain.
In particular, we run across the terms
$$\l A (\delta A^{\gamma-1} {\theta} -b \psi ), {\tilde \psi}\r =
\l \delta A^{\gamma-1} {\theta} -b \psi , {\tilde \psi}\r_1
=\delta\l A^{\gamma-\frac12} {\theta},A^\frac12{\tilde \psi}\r - b{\l \psi , {\tilde \psi}\r}_1$$
and
$$
\l A(c\theta + \delta A^{\gamma-1} {\tilde \psi}),\theta\r
=\l c\theta + \delta A^{\gamma-1}\tilde \psi,\theta\r_1=
c {\|\theta\|}^2_1 + \delta \l A^\frac12\tilde \psi,A^{\gamma-\frac12}\theta\r,
$$
which make sense by virtue of~\eqref{tilpsireg}-\eqref{thetareg}.
\end{proof}

The next result concerns with the invertibility of $L$. Observe that $L$ is a closed operator:
this can be checked directly, or deduced by the subsequent Theorem~\ref{THM-LPh}.
Hence, by the Closed Graph Theorem, if $L^{-1}$ exists, it belongs to ${\mathfrak L}(\H)$ as well.

\begin{proposition}
\label{PropInvertible}
The operator $L$ is invertible if and only if $\gamma \le 1$.
\end{proposition}

\begin{proof}
The operator $L$ is invertible if and only if,
for any  $f =(f_1, f_2, f_3, f_4, f_5) \in \H$,
the equation
$$L z = f$$
admits a unique solution
$z = (\varphi, {\tilde \varphi}, \psi, {\tilde \psi}, \theta) \in \D(L)$.
Componentwise, this translates into
$$
\begin{cases}
\tilde \varphi = f_1, \\
- aA^{\frac12}(A^{\frac12}\varphi + \psi)= \rho_1 f_2, \\
{\tilde \psi} = f_3, \\
A (\delta A^{\gamma-1} {\theta} -b \psi ) -a (A^{\frac12}\varphi + \psi)
=\rho_2 f_4, \\
-A(c\theta + \delta A^{\gamma-1} {\tilde \psi}) = \rho_3 f_5.
\end{cases}
$$
Assume first $\gamma>1$. Choosing $f_3\in H^1$, but not more regular, we see at once from the third equation
of the system that $\tilde \psi\in H^1$ but not more, contradicting~\eqref{tilpsireg}.
Conversely, if $\gamma\leq 1$, the explicit solution $z$ reads
\begin{align*}
&\varphi = -\frac{\rho_1}{ab}(b+aA^{-1}) A^{-1}f_2
+ \frac{\delta^2}{bc}A^{2\gamma -\frac{5}{2}}f_3 + \frac{\rho_2}{b}A^{-\frac{3}{2}}f_4
+ \frac{\rho_3\delta}{bc}A^{\gamma-\frac{5}{2}}f_5, \\
&{\tilde \varphi} = f_1, \\
&\psi = \frac{\rho_1}{b}A^{-\frac{3}{2}}f_2 - \frac{\delta^2}{bc}A^{2\gamma -2}f_3
- \frac{\rho_2}{b}A^{-1}f_4 -\frac{\rho_3\delta}{bc}A^{\gamma-2}f_5, \\
&{\tilde \psi} = f_3, \\
&\theta = - \frac{\delta}{c}A^{\gamma-1}f_3 -\frac{\rho_3}{c}A^{-1}f_5.
\end{align*}
It is apparent that $\tilde \varphi, \psi, \tilde \psi, \theta\in H^1$
and $\varphi\in H^2$. Besides,
$$\delta A^{\gamma-1}\theta - b \psi = - \rho_1 A^{-\frac32}f_2 + \rho_2A^{-1}f_4\in H^2$$
and
$$c\theta + \delta A^{\gamma-1}{\tilde\psi} = -\rho_3 A^{-1}f_5\in H^2.$$
Hence, $z\in\D(L)$.
\end{proof}

\begin{remark}
In fact, when $\gamma\leq 1$, the relation $L^{-1}\in {\mathfrak L}(\H)$
can be easily deduced by the proof above.
\end{remark}

We end the section by discussing the compactness of the embedding $\D(L)\subset \H$
in the case when $A^{-1}$ is a compact operator,\footnote{Clearly,
if the embedding $\D(A)\subset H$ is not compact, the same is true for $\D(L)\subset \H$.}
i.e.\ $\D(A)\Subset H$.

\begin{proposition}
\label{PropCompact}
Assume that $A^{-1}$ is a compact operator.
Then $\D(L)\Subset \H$ if and only if $\gamma<1$.
\end{proposition}

\begin{proof}
First, we provide a counterexample to compactness when $\gamma \geq 1$.
Choose an orthonormal basis $u_n$ of $H^1$
and any two bounded sequences $\varphi_n \in H^2$ and $\tilde\varphi_n \in H^1$.
Then, define the sequence $z_n = (\varphi_n,\tilde\varphi_n, \psi_n, \tilde\psi_n,\theta_n)$,
where
$$\psi_n = \frac{\delta}{b}u_n, \qquad \tilde\psi_n = -\frac{c}{\delta}A^{2-2\gamma}u_n, \qquad \theta_n = A^{1-\gamma}u_n.$$
Since $\gamma \geq 1$, it is readily seen that $\tilde\psi_n,\theta_n \in H^1$.
Moreover, by construction,
$$\delta A^{\gamma-1}\theta_n -b \psi_n = 0, \qquad
c\theta_n + \delta A^{\gamma-1}{\tilde \psi_n} = 0.$$
Thus $z_n$ is a bounded -in the norm of $\D(L)$- sequence in $\D(L)$, whose component $\psi_n$ has no convergent subsequence in $H^1$.

Assume next $\gamma< 1$, and let
$z_n = (\varphi_n, \tilde\varphi_n, \psi_n, \tilde\psi_n, \theta_n)$ be bounded in $\D(L)$.
In particular, $\varphi_n$ is bounded in $H^2 \Subset H^1$, whereas
$\tilde\varphi_n,\tilde\psi_n,\theta_n$ are bounded in $H^1 \Subset H$. Accordingly,
there exist $\varphi\in H^1$ and $\tilde\varphi,\tilde\psi,\theta\in H$ such that, up to a subsequence,
$$\varphi_n  \to \varphi \quad\text{in } H^1$$
and
$$
\tilde\varphi_n \to \tilde\varphi \quad\text{in } H,\qquad
\tilde\psi_n \to \tilde\psi \quad\text{in } H,\qquad
\theta_n  \to \theta \quad\text{in } H.
$$
We are left to prove the convergence $\psi_n\to \psi$ in $H^1$, for some $\psi\in H^1$.
Indeed, knowing that
$$\delta A^{\gamma-1}\theta_n-b \psi_n \quad\text{is bounded in } H^2,$$
we get the convergence, up to a subsequence,
$$\delta A^{\gamma-1}\theta_n-b\psi_n \to\eta \quad\text{in } H^1,$$
for some $\eta\in H^1$.
At the same time, since $\gamma<1$,
$$ A^{\gamma-1}\theta_n\to A^{\gamma-1}\theta\quad\text{in } H^1,$$
so implying the desired convergence.
\end{proof}

\begin{remark}
It is clear that all the results above remain valid for the complexification $\L$
acting on $\H_\CC$ as well,
the only difference being the dissipative estimate~\eqref{Real-DISSIPATIVO},
which becomes
\begin{equation}
\label{Cplx-DISSIPATIVO}
\Re \l \L z, z\r_{\H_\CC} = - c\| \theta\|^2_1, \quad \forall z \in \D(\L).
\end{equation}
\end{remark}

\section{The Contraction Semigroup}
\label{Scinque}

\noindent
The next step is showing that $L$ generates a semigroup.

\begin{theorem}
\label{THM-LPh}
For every fixed $\gamma \in \R$,
the linear operator $L$ is the infinitesimal generator of a strongly continuous
semigroup
$$S(t)=\e^{t L}:\H\to\H$$
of linear contractions.
\end{theorem}

The proof is carried out via an application of the
classical Lumer-Phillips Theorem (see \cite{PAZ}).

\begin{theorem}[Lumer-Phillips]
\label{LP}
The operator $L$ is the infinitesimal generator of a contraction semigroup $S(t)=\e^{t L}$
on $\H$ if and only if
\begin{itemize}
\item[(i)] $L$ is dissipative; and
\smallskip
\item[(ii)] ${\rm ran }(1-L)=\H$.
\end{itemize}
\end{theorem}

Indeed, point (i) is exactly the content of
Proposition~\ref{PropDissipative}. Accordingly, Theorem~\ref{THM-LPh}
follows from the next lemma, establishing (ii).

\begin{lemma}
The operator $1 -L: \D(L) \subset \H \to \H$ is onto.
\end{lemma}

\begin{proof}
For $f= (f_1, f_2, f_3, f_4, f_5) \in \H$,
we look for a solution $z = (\varphi, \tilde\varphi, \psi, \tilde\psi, \theta) \in \D(L)$ to the equation
$$
z - L z = f,
$$
which, componentwise, reads
\begin{align*}
&\varphi - \tilde\varphi = f_1, \\
&\rho_1\tilde\varphi  + a A^\frac12(A^\frac12\varphi + \psi)= \rho_1 f_2, \\
&\psi - \tilde\psi = f_3, \\
&\rho_2\tilde\psi + A (b\psi - \delta A^{\gamma-1} \theta) + a (A^\frac12\varphi + \psi) = \rho_2 f_4, \\
&\rho_3\theta + A(c\theta + \delta A^{\gamma-1} \tilde\psi )= \rho_3 f_5.
\end{align*}
Substituting the first and the third equation of the system above
into the second and the fourth one, respectively, we obtain
\begin{align}
\label{1-LP}
&( \rho_1 + a A) \tilde\varphi + a A^\frac12\tilde\psi = h_1, \\
\label{2-LP}
&a A^\frac12\tilde\varphi + ( \rho_2 + a + bA)\tilde\psi - \delta A^\gamma\theta = h_2,\\
\label{3-LP}
&\delta A^\gamma \tilde\psi +( \rho_3 + c A)\theta = h_3,
\end{align}
where
\begin{align*}
h_1 &= -a A f_1 + \rho_1f_2 - a A^\frac12 f_3\in H^{-1}, \\
h_2 &= -a A^\frac12f_1 - (a + bA) f_3 + \rho_2f_4\in H^{-1}, \\
h_3 &= \rho_3f_5\in H.
\end{align*}
Collecting equations \eqref{1-LP} and \eqref{3-LP}, we learn that
\begin{align}
\label{PHITILDE-LP}
&\tilde\varphi = ( \rho_1 + a A )^{-1}(h_1 - a A^\frac12\tilde\psi), \\
\label{THETA-LP}
&\theta = (\rho_3 + c A)^{-1}(h_3 - \delta A^\gamma \tilde\psi).
\end{align}
Substituting \eqref{PHITILDE-LP}-\eqref{THETA-LP} into \eqref{2-LP},
and exploiting the Functional Calculus of $\A$, we obtain
\begin{equation}
\label{PSITILDE-LP}
\tilde\psi = \int_{\sigma(\A)} \frac{1}{w(t)}\,\d E_\A(t)\,h,
\end{equation}
where
$$
w(t) = \rho_2 + a + bt - \frac{a^2t}{\rho_1 + a t} +\frac{\delta^2 t^{2\gamma}}{\rho_3 + c t}
$$
and
$$
h= h_2 - a A^\frac12( \rho_1 + a A )^{-1}h_1 + \delta A^\gamma(\rho_3 + c A)^{-1}h_3.
$$
Thus,
$$
h\in
\begin{cases}
H^{2- 2\gamma} 	&\text{if } \gamma > \frac{3}{2},\\
H^{-1}			&\text{if } \gamma \leq \frac{3}{2}.
\end{cases}
$$
Observing that $w(t)\geq \rho_2$ for every $t\in\sigma(\A)$ and,
in the limit $t\to\infty$,
$$
w(t) \approx
\begin{cases}
t^{2\gamma-1} 	&\text{if } \gamma > 1,\\
t				&\text{if } \gamma \leq 1,
\end{cases}
$$
we infer that $\tilde\psi\in H^1$, and fulfills \eqref{tilpsireg} as well.
As a byproduct, $\psi=\tilde\psi+f_3\in H^1$.
At this point, we learn from \eqref{PHITILDE-LP}-\eqref{THETA-LP} that
$\tilde\varphi \in H^1$, and $\theta\in H^1$ at least.\footnote{Indeed, one can deduce at this stage
the regularity \eqref{thetareg} only for $\gamma\leq \tfrac32$, whereas for $\gamma>\tfrac32$
equation \eqref{THETA-LP} merely gives $\theta\in H^2$, which is not optimal.}
In order to show that $\varphi=\tilde\varphi+f_1\in H^2$,
we collect \eqref{PHITILDE-LP} and \eqref{PSITILDE-LP}. This entails the explicit expression
\begin{align*}
\varphi &=
\int_{\sigma(\A)} \frac{p_1(t)}{v(t)}\,\d E_\A(t)\,f_1 +
\int_{\sigma(\A)} \frac{p_2(t)}{v(t)}\, \d E_\A(t)\, f_2
+ \int_{\sigma(\A)} \frac{p_3(t)}{v(t)}\,\d E_\A(t)\, f_3\\
&\quad +\int_{\sigma(\A)} \frac{p_4(t)}{v(t)}\, \d E_\A(t)\, f_4 +
\int_{\sigma(\A)} \frac{p_5(t)}{v(t)}\, \d E_\A(t)\, f_5,
\end{align*}
where
$$
v(t)= (\rho_1 + at)(\rho_2 + bt)(\rho_3+ct) + a\rho_1 (\rho_3+ct) + \delta^2 t^{2\gamma}(\rho_1 + at),\\
$$
and
\begin{align*}
p_1(t) &= \rho_1 [(\rho_2 + bt)(\rho_3 + ct) + a(\rho_3 + ct) + \delta^2 t^{2\gamma}], \\
p_2(t) &= \rho_1 [(\rho_2 + a + bt)(\rho_3 + ct) + \delta^2 t^{2\gamma}], \\
p_3(t) &= -at^\frac12[\rho_2(\rho_3 + ct) +\delta^2 t^{2\gamma}], \\
p_4(t) &= -at^\frac12\rho_2(\rho_3 + ct),\\
p_5(t) &= -\rho_3\delta at^{\gamma + \frac12}.
\end{align*}
Note that $v(t)$ is away from zero for $t\in\sigma(\A)$
and, as $t\to\infty$,
$$
v(t) \approx
\begin{cases}
t^{2\gamma+1} 	&\text{if } \gamma > 1,\\
t^3				&\text{if } \gamma \leq 1.
\end{cases}
$$
It is then readily seen that, for every $\gamma\in\R$,
$$\limsup_{t\to\infty}\, t \,\frac{p_\imath(t)}{v(t)}<\infty, \quad\text{for }\imath=1,2,4,5,$$
and
$$\limsup_{t\to\infty}\, t^\frac12 \,\frac{p_3(t)}{v(t)}<\infty.
$$
Recalling in particular that $f_3\in H^1$, we draw the desired conclusion $\varphi\in H^2$.
To finish the proof
we are left to verify the relations
$$\delta A^{\gamma-1}\theta - b\psi\in H^2\and
c\theta + \delta A^{\gamma-1}\tilde\psi\in H^2.$$
Again, by explicit calculations we obtain
\begin{align*}
\delta A^{\gamma-1}\theta - b\psi &=
\int_{\sigma(\A)} \frac{q_1(t)}{v(t)}\,\d E_\A(t)\, (f_1 + f_2) +
\int_{\sigma(\A)} \frac{q_2(t)}{v(t)}\, \d E_\A(t)\, f_3 \\
&\quad + \int_{\sigma(\A)} \frac{q_3(t)}{v(t)}\, \d E_\A(t)\, f_4 +
\int_{\sigma(\A)} \frac{q_4(t)}{v(t)}\, \d E_\A(t)\, f_5,
\end{align*}
and
\begin{align*}
c\theta + \delta A^{\gamma-1}\tilde\psi &=
\int_{\sigma(\A)} \frac{r_1(t)}{v(t)}\,\d E_\A(t)\, (f_1 + f_2) +
\int_{\sigma(\A)} \frac{r_2(t)}{v(t)}\, \d E_\A(t)\, f_3 \\
&\quad + \int_{\sigma(\A)} \frac{r_3(t)}{v(t)}\, \d E_\A(t)\, f_4 +
\int_{\sigma(\A)} \frac{r_4(t)}{v(t)}\, \d E_\A(t)\, f_5,
\end{align*}
with $v(t)$ as above,
\begin{align*}
q_1(t) &= \rho_1 a[\delta^2 t^{2\gamma-\frac12} + b t^\frac12(\rho_3 + ct)], \\
q_2(t) &= \rho_1 a\delta^2 t^{2\gamma-1}- b\rho_2 (\rho_1 + at)(\rho_3 + ct), \\
q_3(t) &= -\rho_2(\rho_1 + at)[\delta^2 t^{2\gamma-1} + b(\rho_3 + ct)], \\
q_4(t) &= \rho_3\delta t^{\gamma - 1}[a\rho_1 + \rho_2 (\rho_1 + at)],
\end{align*}
and
\begin{align*}
r_1(t) &= -\rho_1\rho_3 a\delta t^{\gamma-\frac12}, \\
r_2(t) &= - \rho_3\delta t^{\gamma-1}(\rho_1 a + \rho_1bt + abt^2 ), \\
r_3(t) &= \rho_2\rho_3\delta t^{\gamma-1}(\rho_1 + at), \\
r_4(t) &= c\rho_3[a\rho_1 + (\rho_1 + at)(\rho_2 + bt)] + \rho_3\delta^2 t^{2\gamma - 1}(\rho_1 + at).
\end{align*}
At this point, checking as before the growth orders
of the ratios ${q_\imath(t)}/{v(t)}$ and
${r_\imath(t)}/{v(t)}$ as $t\to\infty$,
the claim follows. The details are left to the reader.
\end{proof}

\begin{remark}
Actually, when $\gamma\leq1$, Theorem~\ref{THM-LPh} can be given a more direct proof. Indeed,
we already know from Proposition~\ref{PropInvertible} that $L^{-1}\in{\mathfrak L}(\H)$.
Hence, $L$ is a closed dissipative operator with
$0\in\rho(L)$, and the conclusion follows
from a slightly modified version of the Lumer-Phillips Theorem (see e.g.\ \cite{LZ}).
\end{remark}

\section{Stability}
\label{Ssei}

\noindent
We provide the result within the assumption  $\D(A) \Subset \H$.

\begin{theorem}
\label{STAB}
If $A^{-1}$ is compact, then the semigroup $S(t)$ is stable for every $\gamma \in \R$.
\end{theorem}

\begin{proof}
Aiming to apply Theorem~\ref{THM-STABLE}, we introduce
the Hilbert space $\V \Subset \H$ defined as
$$\V = H^{p+1} \times H^p \times H^{p+1} \times H^p \times H^p,$$
with $p=p(\gamma)>0$ large enough such that $\V\subset\D(L)$,
endowed with the norm
$$\| (\varphi, {\tilde \varphi}, \psi, \tilde \psi, \theta) \|_{\V}^2 =
a\| A^{\frac12}\varphi + \psi\|^2_p +
b\| \psi\|^2_{p+1} + \rho_1\| \tilde \varphi\|^2_p +
\rho_2\| \tilde \psi\|^2_p + \rho_3\| \theta\|^2_p.$$
Let then $z =(\varphi_0, \tilde\varphi_0, \psi_0, \tilde\psi_0, \theta_0) \in \V$ be arbitrarily fixed.
It is a standard matter to prove that the restriction of $S(t)$ to $\V$
is a contraction semigroup with respect to the norm of $\V$
as well.
Therefore,
$$\|S(t)z\|_\V\leq \|z\|_\V,\quad\forall t\geq 0,$$
showing that the entire orbit of $z$ is bounded in $\V$, hence relatively
compact in $\H$ thanks to the compactness of the embedding.
Assume next
$$\| S(t)z \|_\H = \| z \|_\H,\quad \forall t>0.$$
Exploiting \eqref{Real-DISSIPATIVO}, we get
$$0=\ddt \|S(t)z\|^2_\H = 2\l L S(t)z,S(t)z\r_\H = -2c \| \theta(t)\|_1^2,$$
and so $\theta(t)\equiv 0$. In turn, from the third equation of system~\eqref{SYS}
we infer that $\psi(t)\equiv\psi_0$.
Accordingly, \eqref{SYS} reduces to
$$
\begin{cases}
\rho_1{{\ddot \varphi}} + a A^{\frac12}(A^{\frac12}\varphi + \psi_0) =0,\\
b A \psi_0 + a (A^{\frac12}\varphi + \psi_0) = 0. \\
\end{cases}
$$
The second equation above yields
$$\varphi(t) \equiv -\frac{1}{a}A^{-\frac12}(bA+a)\psi_0,$$
and substituting into the first equation we obtain
$$
A^{\frac32}\psi_0 = 0\quad\Rightarrow\quad \psi_0=0
\quad\Rightarrow\quad \varphi(t)\equiv 0.
$$
Summarizing, we proved that $z=0$, and the claim follows.
\end{proof}

If the embedding $\D(A) \subset \H$ is not compact, the picture becomes less clear, and a comprehensive
result seems out of reach. What we can say in general (see the proof of Theorem~\ref{SEMSTAB1} of the next section)
is that
$$\sigma_{\rm p}(\L) \cap \i \R=\emptyset,\quad\forall\gamma\in\R,$$
no matter whether or not $\D(\L)\Subset\H_\CC$. This is not enough
(albeit necessary) to ensure stability, which would follow
if in addition one knew that
$\sigma_{\rm ap}(\L) \cap \i \R$ is countable (see~\cite{Batty}).
Nevertheless, for $\gamma\in[\tfrac12,1]$, the stability of $S(t)$ is obtained
as a byproduct of Theorem~\ref{SEMSTAB} below.

\section{Semiuniform Stability}
\label{Ssette}

\noindent
We begin by observing that, on account
of Theorem~\ref{THM-SUSTABLE}, $S(t)$ cannot be semiuniformly stable when $\gamma>1$, for its
infinitesimal generator $L$ (and so its complexification $\L$)
is not invertible by Proposition~\ref{PropInvertible}.
The situation is different for $\gamma\leq1$.

\begin{theorem}
\label{SEMSTAB}
If $\gamma \in [\tfrac12,1]$, the semigroup $S(t)$ is semiuniformly stable.
\end{theorem}

\begin{proof}
In light of Theorem \ref{THM-SUSTABLE}, it is sufficient to show that
$\i\R\subset \rho(\L)$. To this end, we appeal to~\cite[Proposition 2.2]{Batty},
which says that if the complexified semigroup $S_\CC(t)$ is bounded (as it is the case), then
$$
\sigma(\L) \cap \i\R = \sigma_{\rm ap}(\L) \cap \i \R.
$$
In other words, it is enough to show that no approximate eigenvalues of the operator $\L$ lie on the imaginary axis.
Indeed, suppose by contradiction that $\i\lambda\in \sigma_{\rm ap}(\L)$, for some $\lambda\in\R$.
Note that $\lambda\neq 0$, for $\L$ is invertible. In this case, there exists a sequence
$z_n = (\varphi_n, {\tilde \varphi}_n, \psi_n, {\tilde \psi}_n, \theta_n) \in \D(\L)$ with
$$
\|z_n\|^2_{\H_\CC} = a{\| \A^{\frac12}\varphi_n + \psi_n\|}^2 +
b{\| \psi_n\|}^2_1 + \rho_1{\| \tilde \varphi_n\|}^2 +
\rho_2{\| \tilde \psi_n\|}^2 + \rho_3{\| \theta_n\|}^2 = 1,
$$
satisfying the relation
$$
\i\lambda z_n - \L z_n \to 0 \quad \text{in } \H_\CC.
$$
Componentwise, we draw the relations
\begin{align}
\label{quattro}
&\i\lambda \varphi_n- \tilde\varphi_n\to 0\quad\text{in } H^1_\CC,\\
\noalign{\vskip0.4mm}
\label{cinque}
&\i\lambda \rho_1\tilde\varphi_n + a \A^{\frac12}(\A^{\frac12}\varphi_{n}+\psi_n)\to 0\quad\text{in } H_\CC,\\
\noalign{\vskip0.7mm}
\label{UNO}
&\i\lambda \psi_n- \tilde\psi_n\to 0\quad\text{in } H^1_\CC,\\
\noalign{\vskip0.4mm}
\label{DUE}
&\i\lambda\rho_2 \tilde\psi_n + \A(b\psi_n-\delta \A^{\gamma-1} \theta_n) +a(\A^{\frac12}\varphi_{n}+\psi_n)
\to 0\quad\text{in } H_\CC,\\
\noalign{\vskip0.7mm}
\label{TRE}
&\i\lambda\rho_3 {\theta}_n + \A (c\theta_n +\delta \A^{\gamma-1} \tilde\psi_n)\to 0\quad\text{in } H_\CC.
\end{align}
By means of \eqref{Cplx-DISSIPATIVO},
$$
\Re\l \i\lambda z_n-\L z_n, z_n\r_{\H_\CC}
=-\Re\l \L z_n, z_n\r_{\H_\CC} = c\|\theta_n\|^2_1,
$$
and since the left-hand side tends to zero as $n\to\infty$, we infer that
\begin{equation}
\label{DISSD}
\theta_n \to 0 \quad\text{in } H^1_\CC.
\end{equation}
Therefore, an application of the operator $\A^{-\frac12}$ to \eqref{TRE} gives
$$\A^{\gamma-\frac12}\,\tilde\psi_n\to 0 \quad\text{in } H_\CC,$$
and since $\gamma\geq\frac12$ this readily implies
\begin{equation}
\label{PSID}
\tilde\psi_n\to 0 \quad\text{in } H_\CC.
\end{equation}
In turn, from \eqref{UNO},
\begin{equation}
\label{mediaPrima}
\psi_n\to 0 \quad\text{in } H_\CC,
\end{equation}
and we deduce from \eqref{DUE} that
$$
b\A^{\frac12}\psi_n+a \varphi_{n} -\delta \A^{\gamma-\frac12}\, \theta_n\to 0 \quad\text{in } H_\CC.
$$
Exploiting \eqref{DISSD} and the assumption $\gamma\leq1$, the relation above reduces to
\begin{equation}
\label{media}
b\A^{\frac12}\psi_n+a\varphi_{n} \to 0 \quad\text{in } H_\CC
\end{equation}
which, by means of \eqref{mediaPrima}, entails
$$
\A^{-\frac12}\, \varphi_n \to 0 \quad\text{in } H_\CC.
$$
At this point, we make use of \eqref{quattro} to get
$$
\A^{-\frac12}\, \tilde\varphi_n \to 0 \quad\text{in } H_\CC.
$$
Hence, by applying $\A^{-\frac12}$ to \eqref{cinque} we find
\begin{equation}
\label{varphiD}
\A^{\frac12}\, \varphi_n + \psi_n \to 0 \quad\text{in } H_\CC,
\end{equation}
and by virtue of \eqref{mediaPrima} we establish the convergence
$$
\varphi_n \to 0 \quad\text{in } H^1_\CC.
$$
As a consequence, \eqref{media} turns into
\begin{equation}
\label{psiD}
\psi_n \to 0 \quad\text{in } H^1_\CC.
\end{equation}
Finally, from \eqref{quattro} we conclude that
\begin{equation}
\label{fin}
\tilde \varphi_n \to 0 \quad\text{in } H_\CC.
\end{equation}
Collecting \eqref{DISSD}-\eqref{PSID} and \eqref{varphiD}-\eqref{fin}, the sought contradiction
is attained.
\end{proof}

The limitation $\gamma\geq\frac12$ plays an essential role in the proof.
Indeed, as seen in the previous section, if $\gamma<\tfrac12$ we cannot even ensure the stability of $S(t)$.
But again, if $\D(A)\Subset\H$ we do have a complete answer.

\begin{theorem}
\label{SEMSTAB1}
If $\A^{-1}$ is compact, then $S(t)$ is semiuniformly stable when $\gamma <\frac12$ as well.
\end{theorem}

\begin{proof}
As in the previous proof, we must show that $\i\R\subset\rho(\L)$.
The difference is that in this case we take advantage of the compact embedding
$\D(\L) \Subset \H_\CC$
ensured by Proposition~\ref{PropCompact}. This allows us to apply
a famous result of T.\ Kato~\cite[Theorem 6.29]{Kat},
stating that
$$
\sigma(\L) = \sigma_{\rm p}(\L),
$$
whenever $\L^{-1}$ is a compact operator.
Therefore, we only have to show that
$$
\sigma_{\rm p}(\L) \cap \i \R=\emptyset.
$$
By contradiction, suppose that $\i\lambda\in \sigma_{\rm p}(\L)$ for some $\lambda\in\R$.
As before, the invertibility of $\L$ forces
$\lambda\neq 0$. Then, there exists a nonnull vector
$z = (\varphi, {\tilde \varphi}, \psi, {\tilde \psi}, \theta)\in\D(\L)$ satisfying
$$\i \lambda z -\L z = 0.$$
In components,
\begin{align}
\label{quattroX}
&\i\lambda \varphi- \tilde\varphi= 0,\\
\noalign{\vskip0.4mm}
\label{cinqueX}
&\i\lambda \rho_1\tilde\varphi + a \A^{\frac12}(\A^{\frac12}\varphi+\psi)= 0,\\
\noalign{\vskip0.7mm}
\label{UNOX}
&\i\lambda \psi- \tilde\psi= 0,\\
\noalign{\vskip0.4mm}
\label{DUEX}
&\i\lambda\rho_2 \tilde\psi + \A (b \psi -\delta \A^{\gamma-1} \theta)+a(\A^{\frac12}\varphi+\psi) = 0,\\
\noalign{\vskip0.7mm}
\label{TREX}
&\i\lambda\rho_3 {\theta} + \A (c\theta +\delta \A^{\gamma-1} \tilde\psi)= 0.
\end{align}
By means of equality \eqref{Cplx-DISSIPATIVO}, we have the identity
$$
0= \Re \l \i\lambda z - \L z, z \r_{\H_\CC} = c \| \theta \|_1^2,
$$
and thus $\theta = 0$. Hence, equation \eqref{TREX} entails $\tilde\psi=0$
and then from \eqref{UNOX} we also obtain $\psi=0$. At this point, from \eqref{DUEX} we infer that
$\varphi=0$, and therefore exploiting \eqref{quattroX} (or \eqref{cinqueX}) we get $\tilde\varphi=0$. The proof is finished.
\end{proof}

\section{Exponential Stability}
\label{Sotto}

\noindent
We now turn our attention to the stronger (and certainly more interesting) notion
of exponential stability. We begin by stating the positive result.
We point out that no compactness assumption on $A^{-1}$ is made.

\begin{theorem}
\label{EXSTAB}
Assume that
$$\chi=0 \and \gamma =  \frac12.$$
Then the semigroup $S(t)$ is exponentially stable.
\end{theorem}

Theorem~\ref{EXSTAB} can be proved via linear semigroup techniques. For instance, a
possibility is to exploit
Theorem~\ref{THM-EXPSTABLE}. However, revisiting the arguments of~\cite{MRR}, a direct proof
can be given, based on the existence of suitable energy functionals. This is the approach
we will follow, which has also the advantage to be exportable to
deal with nonlinear versions of the problem (e.g.\ to prove the existence of bounded absorbing sets).

\begin{proof}[Proof of Theorem \ref{EXSTAB}]
By density, it is enough to show that
$$
E(t) \leq KE(0)\e^{-\kappa t},
$$
where
$$E(t)=\frac12\|S(t)z\|_\H^2$$
is the energy at time $t$ corresponding to the initial datum
$z=(\varphi, {\tilde \varphi}, \psi, \tilde \psi, \theta)\in\D(L)$.
Exploiting~\eqref{Real-DISSIPATIVO}, we deduce the energy equality
\begin{equation}
\label{EnEq}
\ddt E+c\|\theta\|^2_1=0.
\end{equation}
We now define three auxiliary energy functionals:
\begin{align*}
\Lambda_1(t) & = 2\rho_2 \l \dot\psi(t),\psi(t)\r - 2\rho_1\l \dot\varphi(t), \varphi(t)\r,\\
\noalign{\vskip1mm}
\Lambda_2(t) &= \frac{2\rho_2 \rho_3}{\delta}\l A^{-\frac12}\theta(t), \dot\psi(t)\r,\\
\noalign{\vskip1mm}
\Lambda_3(t) &= 2\rho_2\l \dot \psi(t), A^\frac12\varphi(t) + \psi(t)\r - 2\rho_2\l A^\frac12\psi(t), \dot\varphi(t) \r.
\end{align*}
Along this proof, $C>0$ will denote a {\it generic} constant depending only on the structural
parameters of the problem. We will also make use, without explicit mention, of the Poincar\'e
inequality~\eqref{POIN}, as well as of the H\"older and Young inequalities.
The following lemmas hold.

\begin{lemma}
\label{LEMMA-I}
There exists $C_1>0$ such that
$$
\ddt \Lambda_1 + \rho_1\| \dot\varphi\|^2 + b\| \psi \|^2_1 \leq
C_1 \big[ \| \dot\psi\|^2 + \| A^\frac12\varphi + \psi\|^2 + \|\theta\|^2_1 \big].
$$
\end{lemma}

\begin{proof}
By direct computations, the functional $\Lambda_1$ fulfills the identity
$$
\ddt \Lambda_1 + 2\rho_1\| \dot\varphi\|^2 + 2b\| \psi \|^2_1=
2\rho_2\| \dot\psi\|^2 + 2a\| A^\frac12\varphi + \psi\|^2 + 2\delta \l \theta, A^\frac12\psi \r
- 4a \l A^\frac12 \varphi + \psi, \psi \r.
$$
Estimating
$$
2\delta \l \theta, A^\frac12\psi \r - 4a \l A^\frac12 \varphi + \psi, \psi \r \leq b\|\psi \|_1^2 +
C\big[\|A^\frac12 \varphi + \psi\|^2+\|\theta\|^2_1\big],
$$
we are done.
\end{proof}

\begin{lemma}
\label{LEMMA-J}
There exists $C_2>0$ such that, for every $\nu>0$ small,
$$
\ddt \Lambda_2 + \rho_2\| \dot\psi\|^2 \leq
\nu\big[\| \psi\|_1^2 + \|A^\frac12\varphi + \psi \|^2 \big] + \frac{C_2}{\nu}\| \theta \|^2_1.
$$
\end{lemma}

\begin{proof}
The functional $\Lambda_2$ satisfies the differential equality
$$
\ddt \Lambda_2 + 2\rho_2 \| \dot\psi\|^2 = 2\rho_3 \| \theta \|^2 -
\frac{2c\rho_2}{\delta} \l A^\frac12\theta, \dot\psi\r - \frac{2b\rho_3}{\delta} \l \theta, A^\frac12\psi\r
- \frac{2a\rho_3}{\delta} \l A^{-\frac12}\theta, A^\frac12\varphi + \psi\r.
$$
It is immediate to see that
$$
- \frac{2c\rho_2}{\delta} \l A^\frac12\theta, \dot\psi\r \leq \rho_2 \| \dot\psi\|^2 +
C\| \theta\|^2_1.
$$
Moreover, for every $\nu>0$,
$$
 - \frac{2b\rho_3}{\delta} \l \theta, A^\frac12\psi\r
 - \frac{2a\rho_3}{\delta} \l A^{-\frac12}\theta, A^\frac12\varphi + \psi\r
\leq \nu[ \|\psi \|^2_1 + \|A^\frac12\varphi + \psi \|^2 ] + \frac{C}{\nu} \| \theta \|^2_1,
$$
which proves the claim.
\end{proof}

\begin{lemma}
\label{LEMMA-K}
There exists $C_3>0$ such that
$$
\ddt \Lambda_3 + a\| A^\frac12\varphi + \psi\|^2\leq
C_3 \big[ \| \dot\psi\|^2 + \|\theta\|^2_1 \big].
$$
\end{lemma}

\begin{proof}
Taking advantage of the assumption $\chi=0$, we infer that
$$
\ddt \Lambda_3 + 2a\| A^\frac12\varphi + \psi\|^2 =
2\rho_2\| \dot\psi\|^2 + 2\delta \l A^\frac12\theta, A^\frac12\varphi + \psi\r.
$$
Estimating the second term of the right-hand side as
$$
2\delta\l A^\frac12\theta, A^\frac12\varphi + \psi\r \leq a\| A^\frac12\varphi + \psi\|^2 + C\| \theta \|^2_1 ,
$$
we are finished.
\end{proof}

We are now in a position to conclude the proof of the theorem. For $\eps>0$ small,
we define the energy functional
$$
\Lambda(t) =  E(t) + \eps M \bigg[\frac{ a}{2C_1}  \Lambda_1(t)+ \Lambda_3(t)\bigg] + \sqrt{\eps} \Lambda_2(t),
$$
having set
$$
M = 1+ \max \bigg\{\frac{4C_2}{ac},\frac{4C_1 C_2}{abc} \bigg\}.
$$
Collecting Lemmas \ref{LEMMA-I}, \ref{LEMMA-J} and \ref{LEMMA-K},
together with the energy equality \eqref{EnEq}, we obtain
\begin{align*}
\ddt \Lambda &+ \Big(\frac{M\eps a}{2} - \nu \sqrt{\eps}\Big)\|A^{\frac12}\varphi + \psi\|^2 +
\Big(\frac{M \eps  ba}{2C_1}  - \nu \sqrt{\eps}\Big)\|\psi\|_1^2
+ \frac{M\eps a\rho_1}{2C_1}\|\dot\varphi\|^2\\ &
+ \Big(\sqrt{\eps}\rho_2 - \frac{M\eps a}{2}- \eps M C_3\Big)\|\dot\psi\|^2
+ \Big(c-\frac{M\eps a}{2} - \eps MC_3 -\frac{\sqrt{\eps}C_2}{\nu}\Big)\|\theta\|_1^2  \leq 0.
\end{align*}
Therefore, choosing
$$\nu=\frac{2\sqrt{\eps} C_2}{c},
$$
and possibly fixing a smaller $\eps>0$, we end up with
$$
\ddt \Lambda + \eps^2 E \leq 0.
$$
It is also clear that,
for all $\eps>0$ small,
$$\frac12 E(t)\leq \Lambda(t) \leq 2E(t).$$
Hence, the proof follows by an application of the standard Gronwall lemma.
\end{proof}

\section{Lack of Exponential Stability}
\label{Snove}

\noindent
We finally show that the sufficient condition for the exponential stability of $S(t)$ established
in Theorem~\ref{EXSTAB}
is necessary as well. Again, the compactness of $A^{-1}$
is not assumed.

\begin{theorem}
\label{LACKEXSTAB}
If $\chi\neq 0$ or $\gamma\neq \frac12$, then $S(t)$ fails to be exponentially stable.
\end{theorem}

The remaining part of the section is devoted to the proof of Theorem~\ref{LACKEXSTAB}.
First, we need a technical operator-theoretical lemma.

\begin{lemma}
\label{OPTH}
Let $\alpha\in\sigma(\A)$ be fixed,
and let ${\mathcal Q}\subset\R$ be a given bounded set. Then, for every $\eps>0$ small enough,
there exists a unit vector $w_\eps\in H_\CC$
such that the vector
$$\xi_{{q},\eps}=\A^{q} w_\eps-\alpha^{q} w_\eps$$
satisfies the relation
$$\|\xi_{{q},\eps}\|\leq\eps,\quad\forall{q}\in{\mathcal Q}.$$
\end{lemma}

\begin{proof}
For $\eps>0$ small enough, let us consider the interval
$$I_\eps=(\alpha-\eps,\alpha+\eps)\subset\R^+.$$
Since $E_\A(I_\eps)$ is a nonnull projection
(for $\alpha$ belongs to the spectrum), we can select a vector
$$w_\eps\in E_\A(I_\eps)H_\CC
\qquad\text{with}\qquad\|w_\eps\|=1.$$
By the functional calculus of $\A$,
$$\|\xi_{{q},\eps}\|^2=\int_{\sigma(\A)}|t^{q}-\alpha^{q}|^2\,\d\mu^\A_{w_\eps}(t),$$
where, for every Borel set $\Sigma\subset\CC$,
$$\mu^\A_{w_\eps}(\Sigma)=\|E_\A(\Sigma)w_\eps\|^2=\|E_\A(\Sigma)E_\A(I_\eps)w_\eps\|^2
=\|E_\A(\Sigma\cap I_\eps)w_\eps\|^2.
$$
Hence $\mu^\A_{w_\eps}$ is supported on $I_\eps$, and
$$\mu^\A_{w_\eps}(I_\eps)=\|E_\A(I_\eps)w_\eps\|^2=\|w_\eps\|^2=1.
$$
We conclude that
$$
\|\xi_{{q},\eps}\|\leq \sup_{t\in I_\eps}\,|t^{q}-\alpha^{q}|
=\begin{cases}
(\alpha+\eps)^{q}-\alpha^{q} &\text{if } q\geq 1,\\
\alpha^{q}-(\alpha-\eps)^{q} &\text{if } q\in[0,1),\\
(\alpha-\eps)^{q}-\alpha^{q}&\text{if } q<0.
\end{cases}
$$
Thus, for $\eps$ small enough (depending only on $\alpha$ and ${\mathcal Q}$),
$$
\|\xi_{{q},\eps}\|\leq
K\eps,
$$
having set
$$K=\sup_{{q}\in{\mathcal Q}}\,2|{q}|\alpha^{{q}-1}.$$
Up to redefining $\eps$ properly, the proof is finished.
\end{proof}

Select $\alpha_n\in\sigma(\A)$ with $\alpha_n\to\infty$ (this is possible
since $\A$ is unbounded).
By Lemma~\ref{OPTH}, given a positive sequence $\nu_n\to 0$,
there exist $w_n\in H_\CC$ such that the vectors
$$\xi_{{q},n}=\A^{q} w_n-\alpha_n^{q} w_n$$
fulfill the inequality
\begin{equation}
\label{xiconv}
\|\xi_{{q},n}\|\leq\nu_n,\qquad\text{for }q=\gamma,\tfrac12,1.
\end{equation}
Next, we set
$$\hat z_n = (0, c_1 w_n, 0, c_2 w_n,0) \in \H_\CC,$$
where the constants $c_1,c_2$ will be properly chosen in a later moment in such a way that
$$
\|\hat z_n\|_{\H_{\CC}} = 1.
$$
Assume now by contradiction that the semigroup $S(t)$ is exponentially stable.
Then,
for any given sequence $\lambda_n \in \R$ the resolvent equation
$$\i\lambda_n z_n - \L z_n = \hat z_n$$
has a unique solution
$$
z_n = (\varphi_n, {\tilde \varphi}_n, \psi_n, {\tilde \psi}_n, \theta_n) \in \D(\L).
$$
Besides, by Theorem \ref{THM-EXPSTABLE} there is $\eps>0$ such that
\begin{equation}
\label{BBB}
\|z_n\|_{\H_\CC}\leq \frac{1}{\eps}\|\hat z_n\|_{\H_\CC} = \frac{1}{\eps}.
\end{equation}
Namely,
the sequence $z_n$ is bounded.
We will reach a contradiction by showing it is not so.
To this end, we begin to reformulate the resolvent equation above
componentwise. This leads to the system
\begin{align*}
&\i\lambda_n \varphi_n - \tilde\varphi_n = 0, \\
&\i\lambda_n\rho_1 \tilde\varphi_n + a\A^\frac12 (\A^\frac12\varphi_n + \psi_n)= \rho_1c_1w_n,\\
&\i\lambda_n \psi_n - \tilde\psi_n = 0, \\
&\i\lambda_n\rho_2 \tilde\psi_n - \A(\delta \A^ {\gamma-1}\theta_n-b \psi_n)
+ a(\A^\frac12\varphi_n + \psi_n) = \rho_2c_2w_n,\\
&\i\lambda_n \rho_3\theta_n + \A(c \theta_n + \delta \A^{\gamma-1}\tilde\psi_n)=0,
\end{align*}
which, after straightforward calculations, reduces to
\begin{align}
\label{LA1}
-&\rho_1\lambda_n^2 \tilde\varphi_n + a \A^\frac12 (\A^\frac12\tilde\varphi_n + \tilde\psi_n)= \i\lambda_n\rho_1 c_1w_n,\\
\label{LA2}
-&\rho_2 \lambda_n^2 \tilde\psi_n -\A(\i\lambda_n \delta \A^{\gamma-1}\theta_n- b \tilde \psi_n)
+ a(\A^\frac12\tilde\varphi_n + \tilde\psi_n)
 = \i\lambda_n\rho_2 c_2w_n,\\
\label{LA3}
&\i\lambda_n \rho_3 \theta_n +\A(c \theta_n + \delta \A^{\gamma-1}\tilde\psi_n)= 0.
\end{align}
For every $n$, the solution $(\tilde\varphi_n,\tilde\psi_n,\theta_n)$ to \eqref{LA1}-\eqref{LA3}
can be written in the form
\begin{align*}
\tilde \varphi_n = B_n w_n + q^1_n,\\
\noalign{\vskip0.2mm}
\tilde \psi_n = C_n w_n + q^2_n,\\
\theta_n = D_n w_n + q^3_n,
\end{align*}
for some $B_n,C_n,D_n\in\CC$ and some vectors $q_n^\imath$ such that
$$q_n^\imath \perp w_n,\qquad\text{for } \imath=1,2,3.$$
It is then apparent from \eqref{BBB} that
\begin{equation}
\label{boundQ}
\|q_n^\imath\| \leq C,
\end{equation}
where, here and till the end of the proof,
$C\geq0$ stands for a {\it generic} constant depending only on the structural parameters of the problem
(in particular, independent of $n$).
By the same token,
\begin{equation}
\label{bound}
|B_n|\leq C, \qquad |C_n|\leq C,
\qquad |D_n|\leq C.
\end{equation}
Taking the inner product in $H_\CC$ of \eqref{LA1}-\eqref{LA3} and $w_n$, we obtain the system
\begin{align}
\label{LA1x}
-&\rho_1\lambda_n^2 B_n + a [\alpha_n B_n +  \sqrt{\alpha_n}C_n]= f_n + \i\lambda_n \rho_1 c_1,\\
\noalign{\vskip1mm}
\label{LA2x}
-&\rho_2 \lambda_n^2 C_n + b \alpha_n C_n + a[\sqrt{\alpha_n}B_n + C_n]
 -\i\lambda_n \delta \alpha_n^\gamma D_n = g_n + \i\lambda_n\rho_2 c_2 ,\\
 \noalign{\vskip1mm}
\label{LA3x}
&\i\lambda_n \rho_3 D_n + c \alpha_n D_n + \delta \alpha_n^\gamma C_n= h_n,
\end{align}
having set
\begin{align*}
f_n &= -a [B_n \l \xi_{1,n},w_n\r + \l q_n^1 , \xi_{1,n}\r + C_n \l \xi_{\frac12,n}, w_n\r + \l q_n^2,\xi_{\frac12,n}\r],\\
 \noalign{\vskip1mm}
g_n &= -b[ C_n\l \xi_{1,n},w_n\r + \l q_n^2 , \xi_{1,n}\r] -a[B_n\l \xi_{\frac12,n},w_n\r + \l q_n^1,\xi_{\frac12,n}\r]\\
&\quad + \i \lambda_n \delta [D_n \l \xi_{\gamma,n},w_n\r + \l q_n^3 , \xi_{\gamma,n}\r],\\
 \noalign{\vskip1mm}
h_n &= -c [D_n\l \xi_{1,n},w_n\r+\l q_n^3,\xi_{1,n}\r ] - \delta [C_n\l \xi_{\gamma,n},w_n\r + \l q_n^2 , \xi_{\gamma,n}\r].
\end{align*}
By means of \eqref{xiconv} and \eqref{boundQ}-\eqref{bound}, it is readily seen that
\begin{equation}
\label{forz}
|f_n|\leq C \nu_n, \qquad |g_n|\leq C (1+ |\lambda_n|) \nu_n,\qquad |h_n|\leq C \nu_n.
\end{equation}
At this point, we shall distinguish three cases:

\begin{itemize}
\item[(i)] $\gamma>\frac12$.\smallskip
\item[(ii)] $\gamma\leq\frac12$ and $\chi\not=0$.\smallskip
\item[(iii)] $\gamma<\frac12$ and $\chi=0$.\smallskip
\end{itemize}

\medskip
\noindent
$\bullet$ {\bf Cases (i) and (ii).}
Choosing
$$
c_1 = \frac{1}{\sqrt{\rho_1}}, \qquad c_2 = 0, \qquad  \lambda_n = \sqrt{\frac{a\alpha_n}{\rho_1}},
$$
equation \eqref{LA1x} simply becomes
\begin{equation}
\label{pk}
C_n = \frac{f_n}{a\sqrt{\alpha_n}} + \frac{\i}{\sqrt{a}}.
\end{equation}
Substituting \eqref{pk} into \eqref{LA3x}, we find
\begin{equation}
\label{pkpk}
D_n = -\sqrt{\frac{\rho_1}{a}} \frac{\i \alpha_n^{\gamma-\frac12}\delta}{[c \sqrt{\rho_1\alpha_n}+\i \sqrt{a}\rho_3]} + p_n,
\end{equation}
where
$$
p_n = \frac{\sqrt{\rho_1}}{\sqrt{\alpha_n}[c \sqrt{\rho_1\alpha_n}+\i \sqrt{a}\rho_3]}\bigg[h_n
- \frac{\alpha_n^{\gamma-\frac12}\delta f_n}{a}\bigg].
$$
Observe that, by \eqref{forz},
\begin{equation}
\label{forzpn}
|p_n| \leq C|h_n|\alpha_n^{-1} + C|f_n|\alpha_n^{\gamma-\frac32}
\leq C\nu_n\big(\alpha_n^{-1}+\alpha_n^{\gamma-\frac32}\big).
\end{equation}
Finally, plugging \eqref{pk}-\eqref{pkpk} into \eqref{LA2x}, and recalling the definition
of $\chi$, we infer that
\begin{align}
\label{finbi}
B_n &= \bigg[\frac{\chi\rho_2\sqrt{\alpha_n}}{a}-\frac{1}{\sqrt{\alpha_n}}\bigg] C_n
+ \frac{\i\delta\alpha_n^\gamma D_n}{\sqrt{a\rho_1}} + \frac{g_n}{a\sqrt{\alpha_n}}\\
\nonumber
&= \frac{\alpha_n^{2\gamma}\delta^2 c \sqrt{\rho_1}}{a[c^2 \rho_1\alpha_n+ a\rho_3^2]}
- \frac{\i\alpha_n^{2\gamma-\frac12}\delta^2 \sqrt{a}\rho_3}{a[c^2 \rho_1\alpha_n+ a\rho_3^2]}+
\frac{\i}{\sqrt{a}}\bigg[\frac{\chi\rho_2\sqrt{\alpha_n}}{a} - \frac{1}{\sqrt{\alpha_n}}\bigg] + r_n,
\end{align}
where
$$
r_n = \frac{f_n}{a}\bigg[\frac{\chi\rho_2}{a}- \frac{1}{\alpha_n}\bigg]
 + \frac{\i\delta \alpha_n^\gamma p_n}{\sqrt{a\rho_1}} + \frac{g_n}{a\sqrt{\alpha_n}}.
$$
In light of \eqref{forz} and \eqref{forzpn},
$$
|r_n|\leq \frac{C(1+\alpha_n)|f_n|}{\alpha_n} + C\alpha_n^\gamma|p_n| + \frac{C|g_n|}{\sqrt{\alpha_n}}
\leq C\nu_n\big(1+\alpha_n^{\gamma-1}+ \alpha_n^{2\gamma - \frac32}\big).
$$
If $\gamma>\tfrac12$, then
$$
|r_n| = {\rm o} \big(\alpha_n^{2\gamma-1}\big),
$$
and \eqref{finbi} yields
\begin{align*}
|\Re B_n| &= \bigg|\frac{\alpha_n^{2\gamma}\delta^2 c \sqrt{\rho_1}}{a[c^2 \rho_1\alpha_n+ a\rho_3^2]} +
\Re r_n\bigg|\\
\noalign{\vskip1mm}
&\geq \frac{\alpha_n^{2\gamma}\delta^2 c \sqrt{\rho_1}}{a[c^2 \rho_1\alpha_n+ a\rho_3^2]} -|\Re r_n|
\sim \frac{ \alpha_n^{2\gamma-1}\delta^2}{ac\sqrt{\rho_1}}  \to \infty.
\end{align*}
Conversely, if $\gamma\leq\frac12$,
$$
|r_n| = {\rm o} (\sqrt{\alpha_n}),
$$
and since $\chi\not=0$ we learn from \eqref{finbi} that
\begin{align*}
|\Im B_n| &= \bigg|- \frac{\alpha_n^{2\gamma-\frac12}\delta^2 \sqrt{a}\rho_3}{a[c^2 \rho_1\alpha_n+ a\rho_3^2]}+
\frac{1}{\sqrt{a}}\bigg[\frac{\chi\rho_2\sqrt{\alpha_n}}{a} - \frac{1}{\sqrt{\alpha_n}}\bigg] +
\Im r_n\bigg|\\
\noalign{\vskip1mm}
&\geq \bigg|\frac{1}{\sqrt{a}}\bigg[\frac{\chi\rho_2\sqrt{\alpha_n}}{a} - \frac{1}{\sqrt{\alpha_n}}\bigg]
- \frac{\alpha_n^{2\gamma-\frac12}\delta^2 \sqrt{a}\rho_3}{a[c^2 \rho_1\alpha_n+ a\rho_3^2]}\bigg|
-|\Im r_n| \sim \frac{\sqrt{\alpha_n}\chi\rho_2}{a\sqrt{a}}  \to \infty.
\end{align*}
In both cases, we reach the conclusion
$$\|z_n\|_{\H_\CC}\geq\sqrt{\rho_1}\,\|\tilde\varphi_n\|\geq \sqrt{\rho_1}|B_n|\to\infty,
$$
in contradiction with \eqref{BBB}.
\qed

\medskip
\noindent
{\bf Case (iii).}
We choose
$$
c_1 = 0, \qquad c_2 = \frac{1}{\sqrt{\rho_2}},\qquad \lambda_n = \sqrt{\beta_n},
$$
where
$$
\beta_n = \frac{2\rho_2 a\alpha_n + a\rho_1 + a\sqrt{\rho_1^2 + 4\rho_1\rho_2\alpha_n}}{2\rho_1\rho_2}>0.
$$
In particular,
\begin{equation}
\label{asymp}
\lambda_n \sim \sqrt{\frac{a\alpha_n}{\rho_1}}
\end{equation}
and
\begin{equation}
\label{asymp1}
\rho_1\lambda_n^2 -a\alpha_n = \rho_1\beta_n - a\alpha_n \sim a\sqrt{\frac{\rho_1\alpha_n}{\rho_2}}.
\end{equation}
Accordingly, from \eqref{LA1x} we get
$$
B_n = \frac{f_n}{a\alpha_n - \rho_1\beta_n} - \frac{a\sqrt{\alpha_n}C_n}{a\alpha_n - \rho_1\beta_n}.
$$
Moreover, exploiting \eqref{LA3x},
$$
D_n = \frac{h_n}{\i\sqrt{\beta_n}\rho_3 + c\alpha_n} - \frac{\delta \alpha_n^\gamma C_n}{\i\sqrt{\beta_n}\rho_3 + c\alpha_n}.
$$
Hence, substituting the two expressions above into \eqref{LA2x} and
exploiting the assumption $\chi=0$,
we infer that
\begin{equation}
\label{cc}
C_n = \frac{\alpha_n^{1-2\gamma}\sqrt{\rho_2}c}{\delta^2}
+ \frac{\i\sqrt{\beta_n \rho_2}\rho_3}{\delta^2\alpha_n^{2\gamma}} + t_n,
\end{equation}
having set
$$
t_n = \frac{g_n[\i\sqrt{\beta_n}\rho_3 + c\alpha_n]}{\i\sqrt{\beta_n}\delta^2\alpha_n^{2\gamma}}
+ \frac{a\sqrt{\alpha_n}f_n[\i\sqrt{\beta_n}\rho_3 + c\alpha_n]}
{\i\sqrt{\beta_n}\delta^2\alpha_n^{2\gamma}[\rho_1\beta_n-a\alpha_n ]}
+ \frac{h_n}{\delta \alpha_n^\gamma}.
$$
By means of \eqref{forz}, \eqref{asymp} and \eqref{asymp1},
$$
|t_n| \leq C\nu_n\big(\alpha_n^{1-2\gamma}+\alpha_n^{-\gamma}\big).
$$
As $\gamma<\frac12$, we learn that
$$
|t_n| = {\rm o}\big(\alpha_n^{1-2\gamma}\big).
$$
Therefore, from \eqref{cc}, we arrive at
\begin{align*}
|\Re C_n| = \bigg|\frac{\alpha_n^{1-2\gamma}\sqrt{\rho_2}c}{\delta^2} +
\Re t_n\bigg|
\geq \frac{\alpha_n^{1-2\gamma}\sqrt{\rho_2}c}{\delta^2}-| \Re t_n|
\sim \frac{\alpha_n^{1-2\gamma}\sqrt{\rho_2}c}{\delta^2}  \to \infty.
\end{align*}
As before, we end up with
$$\|z_n\|_{\H_\CC}\geq\sqrt{\rho_2}\,\|\tilde\psi_n\|\geq \sqrt{\rho_2}|C_n|\to\infty,
$$
contradicting \eqref{BBB}.
\qed



\end{document}